\documentclass[reqno, 11pt, a4paper]{amsart}
\usepackage[ngerman,english]{babel}
\usepackage{amsmath}
\usepackage{amssymb}
\usepackage{enumerate}
\usepackage{commath}
\usepackage{mdwlist}
\usepackage{mathtools}
\usepackage{ mathrsfs }
\usepackage[arrow, matrix, curve]{xy}
\usepackage[a4paper,total={7in, 8in}]{geometry}
\usepackage{fullpage}
\usepackage{framed}
\usepackage{tikz}
\usetikzlibrary{arrows}
\usepackage{mdframed}
\usepackage{lipsum}
\usepackage{ziffer} 
\usepackage{mathtools}
\usepackage{wrapfig}
\usepackage{graphicx}
\usepackage{tabularx}
\usepackage{verbatim}
\usepackage{amsthm}
\usepackage{commath}
\usepackage{amsfonts}
\usepackage{mathrsfs}
\usepackage{amsopn}
\usepackage{stackrel}
\usepackage{manfnt}
\usepackage{bbm}
\usepackage{fancyhdr}
\usepackage{centernot}
\usepackage{stmaryrd}
\usepackage{comment}
\usepackage{cancel}
\usepackage{hyperref}
\usepackage{kantlipsum}
\allowdisplaybreaks

\theoremstyle{definition}
\newtheorem{defn}{Definition}[section]

\newmdtheoremenv{Definition}[defn]{Definition}

\newtheorem{remark}[defn]{Remark}

\newtheorem*{corol*}{Corollary}
\newtheorem*{remark*}{Bemerkung}

\newtheorem*{exmp*}{Example}

\newtheorem*{subalg*}{\textsf{Sub-Algorithmus}}

\theoremstyle{plain}
\newtheorem{thm}[defn]{Theorem}
\newmdtheoremenv{theo}[defn]{Theorem}
\newmdtheoremenv{Lemma}[defn]{Lemma}
\newmdtheoremenv{Korollar}[defn]{Corollary}
\newtheorem*{thm*}{Theorem}
\newtheorem{lemma}[defn]{Lemma}

\newtheorem{prop}[defn]{Proposition}
\newtheorem{corol}[defn]{Corollary}

\newtheorem*{thmen*}{Theorem}

\newtheorem*{corolen*}{Corollary}

\numberwithin{equation}{section}

\newcommand{\R}{\ensuremath{\mathbb{R}}}

\newcommand{\N}{\ensuremath{\mathbb{N}}}

\newcommand{\E}{\ensuremath{\mathbb{E}}}

\newcommand{\ind}{\ensuremath{\mathbbm{1}}}
\newcommand{\defeq}{\ensuremath{\vcentcolon}=}
\newcommand{\eps}{\epsilon}

\newcommand{\PP}{\ensuremath{\mathbb{P}}}

\newcommand{\ddd}{\ensuremath{\text{d}}}

\newcommand{\vertiii}[1]{\vert\kern-0.25ex\vert\kern-0.25ex\vert #1 
    \vert\kern-0.25ex\vert\kern-0.25ex\vert}

\begin{document}

\title{The compact interface property for the stochastic heat equation with seed bank }


\date{\today}

\begin{abstract}
We investigate the compact interface property in a recently introduced variant of the stochastic heat equation that incorporates dormancy, or equivalently seed banks. There individuals can enter a dormant state during which they are no longer subject to spatial dispersal and genetic drift. This models a state of low metabolic activity as found in microbial species. Mathematically, one obtains a memory effect since mass accumulated by the active population will be retained for all times in the seed bank. This raises the question whether the introduction of a seed bank into the system leads to a qualitatively different behaviour of a possible interface. Here, we aim to show that nevertheless in the stochastic heat equation with seed bank compact interfaces are retained through all times in both the active and dormant population. We use duality and a comparison argument with partial functional differential equations to tackle technical difficulties that emerge due to the lack of the martingale property of our solutions which was crucial in the classical non seed bank case.
\end{abstract}
\author{Florian Nie}
\address{Technische Universit\"at Berlin}
\curraddr{Strasse des 17. Juni 136, 10623 Berlin}
\email{nie@math.tu-berlin.de}

\thanks{}

\maketitle

\section{Introduction and main result}
One of the simplest spatial models for the evolution of the frequency of a bi-allelic population under the influence of random genetic drift is given by
\begin{align} \label{eq:StochasticHeat}
    \partial_t u(t,x) = \frac{\Delta}{2} u(t,x) +\sqrt{u(t,x)(1-u(t,x))} \dot W(t,x),
\end{align}
where $W=(W(t,x))_{t\geq 0, x\in \R}$ is a Gaussian white noise process, which is called the stochastic heat equation with Wright-Fisher noise introduced by Shiga in \cite{S88}. Here, $u(t,x)$ models the frequency of one of the two types at space time point $(t,x) \in [0,\infty[ \times \R$. Heuristically, one can interpret the model as individuals migrating among a continuum of colonies in a diffusive way. Moreover, reproduction is subject to random genetic drift with variance $u(t,x)(1-u(t,x))$.

This model has been studied extensively in the past (see e.g.\ \cite{T95} and \cite{M95}) and it turns out that one of the remarkable properties that distinguishes the stochastic model from the deterministic heat equation is the compact interface property. In order to introduce this property we define for a function $f\colon \R \to \R$ with $0\leq f(x)\leq 1$ for all $x \in \R$, $f(x) \to 1$ as $x \to -\infty$ and $f(x) \to 0$ as $x \to \infty$ the left and right edge as follows:
\begin{align*}
    L(f) &=\inf\lbrace x\in \R \vert \,f(x)<1\rbrace, \\
    R(f) &=\sup\lbrace x \in \R\vert \, f(x)>0 \rbrace.
\end{align*}
We say that Equation \eqref{eq:StochasticHeat} exhibits the compact interface property if for some initial condition $0\leq u(0,\cdot)\leq 1$ with $L(u(0,\cdot))>-\infty$ and $ R(u(0,\cdot))< \infty$ it follows that 
\begin{align*}
    L(u(t,\cdot))&>-\infty,\\
    R(u(t,\cdot))&<\infty
\end{align*}
for all $t>0$ almost surely. For Equation \eqref{eq:StochasticHeat} this was shown in \cite{T95} but extensive research regarding this property has also been carried out in the context of super Brownian motion (cf. \cite{Iscoe88}) and more general equations like the symbiotic branching model (see \cite{S94}). In contrast, for a solution $\tilde u$ of the deterministic heat equation we of course would have 
$$L(\tilde u(t,\cdot))=-\infty, \, R(\tilde u(t,\cdot))=\infty,$$
even when started from the same initial condition $u(0,\cdot)$. Biologically, this can be interpreted as a finite zone to which the entire interaction between the two types is confined to.

Recently, in the context of microbial species, an additional evolutionary mechanism in the form of {\em dormancy}, or equivalently {\em seed banks}, has raised considerable attention in population genetics (see e.g.\ \cite{LJ11}, \cite{SL18}). Mathematically, this mechanism has been incorporated into the classical (non-spatial) Wright Fisher model and investigated in \cite{BEGCKW15} and \cite{BGCKW16}. There dormancy and resuscitation are modeled in the form of classical migration between an active and an inactive state. Corresponding discrete-space population genetic models have also very recently been introduced in \cite{GHO20}.

For the case of a continuous spatial structure, in \cite{SFKPP} the following system of SPDEs was established to allow individuals to retreat into a seed bank where spatial dispersal and random genetic drift are absent:
\begin{align} \label{eq:OnOffStochasticHeat}
    \partial_t u(t,x)&= \frac{\Delta}{2}u(t,x) +c( v(t,x)-u(t,x)) + \sqrt{u(t,x)(1-u(t,x))} \dot W(t,x)\nonumber\\
    \partial_t v(t,x) &= c'(u(t,x)-v(t,x)).
\end{align}

Here $c,c'>0$ are the seed bank migration rates. This equation admits unique in law weak solutions when started from Heaviside initial conditions and satisfies a moment duality to a system of ``on/off" coalescing Brownian motions. This object is a coalescing Brownian motion where spatial movement and coalescence can be switched on and off at rates $c'$ and $c$, respectively. \\
It also turns out that the following reformulation of Equation \eqref{eq:OnOffStochasticHeat} as a stochastic partial delay differential equation is crucial in both proofs and heuristic considerations:
\begin{align} \label{eq:OnOffStochasticHeatDelay}
    \partial_t u(t,x)&= \frac{\Delta}{2}u(t,x) +c\left( e^{-c't}v(0,x) + \int_0^t e^{-c'(t-s)} u(s,x) \, \ddd s-u(t,x)\right)\nonumber\\
    &\qquad + \sqrt{u(t,x)(1-u(t,x))} \dot W(t,x),\nonumber\\
     v(t,x) &= e^{-c't}v(0,x) + \int_0^t e^{-c'(t-s)} u(s,x) \, \ddd s.
\end{align}

\medskip
Our goal now is to investigate whether in this seed bank model the compact interface property also holds. Note that from the Delay Equation \eqref{eq:OnOffStochasticHeatDelay} it is immediately obvious that the interface of the dormant component $v$ is increasing in time since mass the active population $u$ accumulated is retained through all times. One may think of this as a memory effect introduced by the seed bank. This is in stark contrast to the classical non seed bank case where the interface can shrink and move freely in space and time. Similarly, this memory effect leads to an upwards drift for the active component $u$ albeit the situation is less clear compared to the dormant population due to the presence of the noise. Intuitively, this would then suggest that the interface becomes larger after introduction of a seed bank, raising the question whether it becomes too large to retain its compactness.

In this paper we show that this is indeed \textit{not} the case and the compact interface property holds at all times almost surely. In the process of doing so we will also provide on/off versions of well-known statements like the Feynman-Kac formula.

For the proof of the main result we use a comparison argument with deterministic differential equations originating from the theory of super Brownian motion as in \cite{T95}, \cite{D89} and \cite{Etheridge04}. Note however that their arguments rely heavily on the fact that the corresponding SPDE solutions are martingales. This is not true in our case due to the seed bank drift term. This technical difficulty has previously been tackled for the stochastic FKPP equation in \cite{M19} by using the Girsanov theorem for SPDE. The seed bank drift term does however not satisfy the prerequisites for the Girsanov theorem so that we resort to duality and comparison with a partial \textit{functional} differential equation instead of a classical PDE to overcome these difficulties.

The following theorem is the main result of this paper:
\begin{thm} \label{thm:main}
Let $u_0=v_0= \ind_{]-\infty,0]}$ and $(u,v)$ be the solution of Equation \eqref{eq:OnOffStochasticHeat} with $c=c'\geq 1$ corresponding to these initial conditions. Then, almost surely, we have 
\begin{align*}
    L(u(t,\cdot))&>- \infty,  &L(v(t,\cdot))>-\infty, \\
    R(u(t,\cdot))&< \infty,  &R(v(t,\cdot))<\infty, 
\end{align*}
for all $t\geq 0$. 
\end{thm}

\medskip
The result of this paper seems to open up some interesting and challenging lines of further research. 

For example, in \cite{T95} and \cite{B16} it was shown that the interface of the classical stochastic heat equation and the symbiotic branching model have non-trivial scaling limits. This raises the question whether this is still true for the stochastic heat equation with seed bank and how the limit compares to the previous ones. We would like to point out however that showing tightness for Equation \eqref{eq:OnOffStochasticHeat} even in a weaker topology, like the Meyer Zheng topology, seems to be more challenging than in the previous cases due to the lack of the martingale property for the solutions.

 Moreover, it would also be interesting to investigate whether the result of this paper can be extended to the stochastic FKPP equation with seed bank. This would enable more in-depth study of the``right marker speed"
 $\lim_{t \to \infty} R(u(t,\cdot))/t$
which was shown to exist and be strictly positive for the classical stochastic FKPP equation in \cite{C05}.

\section{The stochastic heat equation with seed bank} \label{sec:StoHeatSeedBank}

We recall some basic results regarding Equation \eqref{eq:OnOffStochasticHeat} from \cite{SFKPP}. The proofs and further additional motivation may be found there as well.

\begin{thm}\label{thm:existence}
Let $u_0=v_0= \ind_{]-\infty,0]}$. Then there exists a weak solution $(u,v)$ of Equation \eqref{eq:OnOffStochasticHeat} with $u(t,\cdot) \in C(\R,[0,1])$ and $v(t,\cdot) \in B(\R,[0,1])$ for all $t>0$ almost surely which is unique in law and has the following integral representation:
\begin{align}
    u(t,x) &= G_t u_0(x) + c\int_0^t \int_\R G(t-s,x,y) (v(s,y)-u(s,y)) \, \ddd x\, \ddd s\\
    &\qquad + \int_0^t \int_\R G(t-s,x,y) \sqrt{u(s,y)(1-u(s,y))} \, W(\ddd x, \ddd s), \\
    v(t,x)&=v_0(x)+c'\int_0^t u(s,x)-v(s,x) \, \ddd x,
\end{align}
where $G(t,x,y)= \frac{1}{\sqrt{2 \pi t}}\exp\left(-\frac{(x-y)^2}{2t}\right)$ is the heat kernel for $t\geq 0$ and $ x,y \in \R$ and $(G_t)_{t \geq 0}$ denotes the heat semigroup given by
$$G_tf(x) =\int_\R G(t,x,y) f(y) \, \ddd x$$
for $f \in B(\R)$ and $x \in \R$.
\end{thm}

As mentioned before this equation has a dual process which is defined as follows:
\begin{defn}
We denote by $M=(M_t)_{t \geq 0}$ an on/off coalescing Brownian motion taking values in $\bigcup_{k \in \N_0} \left(\R \times \lbrace \boldsymbol{a},\boldsymbol{d}\rbrace \right)^k$ starting at $M_0= ((x_1,\sigma_1), \cdots, (x_n,\sigma_n ))\in\left(\R \times \lbrace \boldsymbol{a}, \boldsymbol{d}\rbrace\right)^n$ for some $n \in \N$. Here the marker $\boldsymbol{a}$ (resp. $\boldsymbol{d}$) means that the corresponding particle is active (resp. dormant).
The process evolves according to the following rules:
\begin{itemize}
    \item Active particles, i.e.\ particles with the marker $\boldsymbol{a}$, move in $\R$ according to independent Brownian motions.
    \item Pairs of active particles coalesce according to the following mechanism:
    \begin{itemize}
        \item We define for each pair of particles labelled $(\alpha, \beta)$ their intersection local time $L^{\alpha, \beta}=(L^{\alpha , \beta}_t)_{t \geq 0}$  as the local time of $M^\alpha-M^\beta$ at $0$ which we assume to only increase whenever both particles carry the marker $\boldsymbol{a}$.
        \item Whenever the intersection local time exceeds the value of an independent exponential clock with rate $1$, the two involved particles coalesce into a single particle.
    \end{itemize}
    \item Independently, each active particle switches to a dormant state at rate $c$ by switching its marker from $\boldsymbol{a}$ to $\boldsymbol{d}$.
    \item Dormant particles do not move or coalesce.
    \item Independently, each dormant particle switches to an active state at rate $c'$ by switching its marker from $\boldsymbol{d}$ to $\boldsymbol{a}$.
\end{itemize}
Moreover, denote by $I=(I_t)_{t\geq 0}$ and $J=(J_t)_{t \geq 0}$ the (time dependent) index set of active and dormant particles of $M$, respectively, and let $N_t$ be the random number of particles at time $t\geq 0$ so that $M_t=(M^1_t, \cdots , M^{N_t}_t)$. 
\end{defn}

Next we recall a moment duality between the solution to Equation \eqref{eq:OnOffStochasticHeat} and the previously defined on/off coalescing Brownian motion $M=(M_t)_{t\geq 0}$.

\begin{thm} \label{thm:duality}
 Let $(u,v)$ be a solution to the system \eqref{eq:OnOffStochasticHeat} with initial conditions $u_0,v_0 \in B (\R)$. Then we have for any initial state $M_0 =((x_1, \sigma_1), \cdots , (x_n, \sigma_n)) \in \left(\R \times \lbrace \boldsymbol{a}, \boldsymbol{d} \rbrace\right)^n$, $n\in \N$ and $t \geq 0$
 \begin{align*}
     \E\left[ \prod_{\beta \in I_0} u(t,M^\beta_0)  \prod_{\gamma \in J_0} v(t, M^\gamma_0) \right]&= \E\left[ \prod_{\beta \in I_t} u_0(M^\beta_t)  \prod_{\gamma \in J_t} v_0( M^\gamma_t) \right].
 \end{align*}
\end{thm}

Finally, we provide a delay representation of the $v$ component in terms of the $u$ component which will become useful later on.

\begin{thm}\label{thm:DelayRepresentation}
 Let $(u,v)$ be a solution to the system \eqref{eq:OnOffStochasticHeat} with initial conditions $u_0,v_0 \in B (\R)$. Then we have
 \begin{align*}
     v(t,x)= e^{-ct} v_0(x)+e^{-ct} \int_0^t e^{cs} u(s,x) \, \ddd s.
 \end{align*}
\end{thm}

\section{Proof of Theorem \ref{thm:main}} \label{sec:Existence}

\begin{prop} \label{prop_int_ex}
Let $(u,v)$ be a solution of \eqref{eq:OnOffStochasticHeat} with initial conditions $u_0=v_0=\ind_{]-\infty, 0]}$ and assume that $c=c'$. Then, for all $t>1$, there exists some $b_0 \geq 0$ and a map $\eta(t,b)$ integrable in $b$ on $(b_0,\infty)$ such that 
\begin{equation*}
\PP\left(\sup_{0\leq s \leq t} \sup_{x \in [b,\infty[} u(s,x) >0 \right)\leq \eta(t,b)
\end{equation*}
for all $b \geq b_0$.
\end{prop}

\begin{proof}
The proof follows the general structure of \cite[Proposition 3.2]{T95}. Let $b>0$ be arbitrary but fixed. We begin by taking some bounded $\psi \in L^2(\R) \cap C^1(\R)$ such that $0\leq \psi \leq 1$ and $\lbrace x \colon \psi(x)>0 \rbrace=(0,\infty)$. Moreover, define $\psi_b(x)\defeq\psi(x-b)$ and the stopping times
\begin{align*}
 \tau_b&\defeq \inf \lbrace t\geq 0 \colon \exists x \geq b/2 \text{ s.t. } u(t,x)\geq 1/2 \rbrace,\\
 \sigma_b &\defeq \inf \lbrace t\geq 0 \colon \langle u(t,\cdot), \psi_b \rangle >0\rbrace.
\end{align*}
The main idea of the proof is to show that for each $t>0$ there exists some map $\eta(t,b)$ with the aforementioned properties such that
\begin{align*}
    \PP(\sigma_b \leq t ) \leq \eta(t,b)
\end{align*}
as the statement of the proposition follows immediately.

Next, fix $t>1$, $\lambda>0$ and apply Ito's formula to see that for $0\leq s\leq t$

\begin{align*}
&\exp\left( -\langle u(s,\cdot), h^\lambda(s, \cdot) \rangle - \langle v(s, \cdot), k^\lambda(s,\cdot) \rangle -\lambda\int_0^s \langle u(r, \cdot), \psi_b \rangle \, \ddd r \right) \\
&\qquad= \exp\left(-\langle u_0, h^\lambda(0,\cdot ) \rangle-\langle v_0, k^\lambda(0,\cdot ) \rangle \right)\\
&\qquad\qquad + \int_0^s \exp\left( -\langle u(s',\cdot), h^\lambda(s', \cdot) \rangle - \langle v(s', \cdot), k^\lambda(s',\cdot) \rangle -\lambda\int_0^{s'} \langle u(r, \cdot), \psi_b \rangle \, \ddd r \right)\\
& \qquad\qquad \qquad \times \left( \langle u(s',\cdot), -\frac{\Delta}{2} h^\lambda(s',\cdot)-\partial_{s'} h^\lambda(s', \cdot)-\lambda \psi_b \rangle - c \langle v(s',\cdot )-u(s', \cdot), h^\lambda(s',\cdot) \rangle \right. \\
&\quad\qquad\qquad \qquad \left.  - \langle v(s',\cdot) , \partial_{s'}  k^\lambda(s',\cdot) \rangle -c \langle u(s', \cdot)-v(s',\cdot ), k^\lambda(s',\cdot ) \rangle  \right.\\
&\quad\qquad\qquad \qquad+\left. \frac{1}{2}\langle u(s', \cdot )(1-u(s',\cdot)) ,( h^\lambda(s', \cdot))^2 \rangle \, \ddd s' \right)+ H_s\\
&\qquad=\exp\left(-\langle u_0, h^\lambda(0,\cdot ) \rangle-\langle v_0, k^\lambda(0,\cdot ) \rangle \right)\\
&\qquad\qquad + \int_0^s \exp\left( -\langle u(s',\cdot), h^\lambda(s', \cdot) \rangle - \langle v(s', \cdot), k^\lambda(s',\cdot) \rangle -\lambda\int_0^{s'} \langle u(r, \cdot), \psi_b \rangle \, \ddd r \right) \\
&\qquad\quad \qquad  \times \left\langle- \frac{1}{4} u(s', \cdot) + \frac{1}{2} u(s', \cdot )(1-u(s',\cdot)) , (h^\lambda(s', \cdot))^2 \right\rangle \, \ddd s'  +H_s,
\end{align*}
where $H$ is a continuous local martingale and we choose $(h^\lambda,k^\lambda)$ as the time reversed versions of the solution $(\phi^\lambda,\varphi^\lambda)$\footnote{This means that we set $h^\lambda(s,x)=\phi^\lambda(t-s,x)$ and $k^\lambda(s,x)=\varphi^\lambda(t-s,x)$ for $(s,x)\in [0,t]\times \R$.} to the system of PDEs on $[0,t]\times \R$ given by
\begin{align}
	\partial_s \phi^\lambda (s,x) &= \frac{\Delta}{2} \phi^\lambda(s,x) -\frac{1}{4} (\phi^\lambda(s,x))^2 +c (\varphi^\lambda(s,x) -\phi^\lambda(s,x) )+ \lambda \psi_b(x),\nonumber\\
    \partial_s \varphi^\lambda(s,x) &= c(\phi^\lambda(s,x)-\varphi^\lambda(s,x)) \label{eq:Delaysystem}
\end{align}
with initial condition $\phi^\lambda(0,\cdot)=\varphi^\lambda(0,\cdot) \equiv 0$.
By integration by parts (see Theorem \ref{thm:DelayRepresentation}), we see that
\begin{align}\label{eq:DelayRepresentation}
    \varphi^\lambda(s,x) = ce^{-cs}\int_0^s e^{cs'}\phi^\lambda(s',x) \, \ddd s'
\end{align}
for any $0\leq s \leq t$.
Hence, substituting this into the original equation, we may interpret Equation \eqref{eq:Delaysystem} as a partial functional differential equation with the same initial condition via
\begin{align}\label{eq:Delaypde}
	\partial_s \phi^\lambda (s,x) &= \frac{\Delta}{2} \phi^\lambda(s,x) -\frac{1}{4} (\phi^\lambda(s,x))^2 +c \left(ce^{-cs}\int_0^s e^{cs'}\phi^\lambda(s',x) \, \ddd s' -\phi^\lambda(s,x)\right )+ \lambda \psi_b(x).
\end{align}
By Lemma \ref{lemma:DelayExistence}, this equation has a unique positive $C^{1,\infty}_b([0,T] \times \R)\cap B([0,T],L^2( \R))$-valued solution. Thus, $H$ given by 
\begin{align*}H_s&= \int_0^s \int_\R \exp\left( -\langle u(s',\cdot), h^\lambda(s', \cdot) \rangle - \langle v(s', \cdot), k^\lambda(s',\cdot) \rangle -\lambda\int_0^{s'} \langle u(r, \cdot), \psi_b \rangle \, \ddd r \right)\\
&\qquad \sqrt{u(s',y)(1-u(s',y))} h^\lambda (s',y)\, W(\ddd y,\ddd s') 
\end{align*}
for $0\leq s\leq t$ is actually a true martingale.

Moreover, we note that outside the support of $\psi_b$ on $(-\infty, b) $ the map given by
\begin{align*}
	\xi(s,x)=\frac{\alpha}{(x-b)^2} \ind_{(-\infty, b)}(x)
\end{align*}
for $0\leq s \leq t$ and $ x \in \R$ satisfies the partial functional differential inequality 
\begin{align}\label{eq:Delaypdi}
\partial_s \xi (s,x) &\geq \frac{\Delta}{2} \xi(s,x) -\frac{1}{4} (\xi(s,x))^2 +c \left(ce^{-cs}\int_0^s e^{cs'}\xi(s',x) \, \ddd s' -\xi(s,x)\right )+ \lambda \psi_b(x).
\end{align}
Indeed, we have for $x<b$
\begin{align*}
    &\partial_s \xi (s,x) -\frac{\Delta}{2}\xi(s,x) +\frac{1}{4} \xi^2(s,x)  -c \left(ce^{-cs}\int_0^s e^{cs'}\xi(s',x) \, \ddd s' -\xi(s,x)\right )- \lambda \psi_b(x) \\
    & \qquad = \frac{\alpha(\alpha-12)}{4(x-b)^4} - c(1-e^{-cs}) \frac{\alpha}{(x-b) ^2} +c \frac{\alpha}{(x-b) ^2}\\
    &\qquad \geq 0
\end{align*}

\noindent if $\alpha$ is large enough. By a comparison theorem (e.g.\ a slight modification of \cite[Theorem 4.II]{B63}\footnote{Note that for each $\lambda>0$ the map $h^\lambda$ is uniformly bounded and hence for $\eps>0$ small enough we will have $\xi(s,x)=\frac{\alpha}{(x-b)^2}\geq h^\lambda (s,x)$ on $[0,t]\times [ b-\eps ,b[$. Moreover, since $\xi$ is bounded on $[0,t]\times ]-\infty, b-\eps]$ we only require the Lipschitz condition on a compact interval. We can thus apply the comparison theorem on $[0,t]\times ]-\infty, b-\eps]$ to get \eqref{eq:bound1}.}) and Equation \eqref{eq:DelayRepresentation}, this implies that
\begin{align} \label{eq:bound1}
	h^\lambda(s,x) &\leq \frac{\alpha}{(x-b)^2},\\
	k^\lambda(s,x) &\leq c  e^{-cs} \int_0^s e^{cs'}\frac{\alpha}{(x-b)^2} \, \ddd s'\\
	&\leq \frac{\alpha}{(x-b)^2}
\end{align}
for all $s\leq t,x<b$.

Then, on the set $\lbrace \sigma_b< t \wedge \tau_b \rbrace$ we have 
\begin{align*}
\langle u(t\wedge \tau_b,\cdot), h^\lambda(t\wedge \tau_b, \cdot) \rangle  +\langle v(t\wedge \tau_b,\cdot), k^\lambda(t\wedge \tau_b, \cdot) \rangle +\lambda \int_0^{t \wedge \tau_b} \langle u(r, \cdot), \psi_b \rangle \, \ddd r \to \infty
\end{align*}
as $\lambda \to \infty$. This implies, since $h^\lambda$ is increasing in $\lambda$ and $0\leq u\leq \frac{1}{2}$ on $[0,\tau_b \wedge t] \times [b/2,\infty[$, that
\begin{align*}
&\PP(\sigma_b\leq t \wedge \tau_b )\\
&\quad \leq\lim_{\lambda \to \infty}\E\left[ 1- \exp\Bigg(-\langle u(t \wedge \tau_b, \cdot), h^\lambda(t \wedge \tau_b, \cdot) \rangle -\langle v(t \wedge \tau_b, \cdot), k^\lambda (t \wedge \tau_b, \cdot) \rangle\right.\\
&\qquad\quad \left.\left. - \lambda\int_0^{t \wedge \tau_b} \langle u(s, \cdot), \psi_b \rangle \, \ddd s \right)  \right] \\
&\quad \leq 1-\E \left[\exp\left(-\langle u_0, h^\infty(0,\cdot ) \rangle-\langle v_0, k^\infty(0,\cdot ) \rangle \right)\right]+\E \left[   \int_0^{t \wedge \tau_b} \left\langle \frac{1}{4} u(s,\cdot)\ind_{]-\infty ,b/2[ }, (h^\infty(s,\cdot))^2 \right\rangle \, \ddd s  \right] \\
&\quad \leq \langle u_0, h^\infty(0,\cdot ) \rangle+\langle v_0, k^\infty(0,\cdot ) \rangle +\E \left[   \int_0^{t \wedge \tau_b} \left\langle \frac{1}{4} u(s,\cdot)\ind_{]-\infty ,b/2[ }, (h^\infty(s,\cdot))^2 \right\rangle \, \ddd s  \right], 
\end{align*}
where $h^\infty \defeq \lim_{\lambda \to\infty} h^\lambda$, which exists on $]- \infty , b/2] $ by the bound \eqref{eq:bound1}, and $k^\infty \defeq \lim_{\lambda \to\infty} k^\lambda$, which exists by Equation \eqref{eq:DelayRepresentation} and the dominated convergence theorem. Thus, using \eqref{eq:bound1} and Lemma \ref{lemma:initialbound}, we see for some constant $C_1(t)>0$ that if $b \geq \sqrt{t}$
\begin{align*}
&\PP(\sigma_b<t\wedge \tau_b)\\
&\qquad \leq C_1(t) \left(\int_{-\infty}^0 \exp\left(-\frac{(x-b)^2}{20t} \right)  \, \ddd x +  \int_{-\infty}^{b/2} \frac{\alpha^2}{(x-b)^4} \, \ddd x  \right)\\
&\qquad \leq C_1(t) \left( \exp(-b^2/20t) + \frac{8 \alpha^2}{3 b^3} \right),
\end{align*}
where we used the standard Gaussian tail bound
\begin{align}
    \frac{1}{\sqrt{2\pi}} \int_x^\infty e^{-y^2/2} \, \ddd y \leq \frac{e^{-x^2/2}}{x \sqrt{2 \pi}}
\end{align}
for $x \geq 0$.
Hence, if we now show that for some $C(t)>0$
\begin{align*}
\PP(\tau_b \leq t) \leq C(t)  \exp\left(-\frac{b^2}{8t} \right),
\end{align*}
our claim is proven.

For this purpose, we note that 
\begin{align*}
\PP(\tau_b\leq t) &\leq\PP\left( \exists x \in [b/2,\infty[,s \leq t \colon G_s u(0,\cdot)(x) +  N_s(x) \geq \frac{1}{2}  \right),
\end{align*}
where for $t\geq 0$
\begin{align*}
N_t(x) &\defeq \int_0^t \int_{\R} G(t-s,x,y) c(v(s,y)+u(s,y)) \, \ddd y \, \ddd s \\
&\qquad + \int_0^t \int_\R G(t-s,x,y) \sqrt{(1-u(s,y))u(s,y)} \, W(\ddd s, \ddd y)
\end{align*}
and $(G_t)_{t \geq 0}$ denotes the heat semigroup given by
\begin{align} \label{eq:HeatOperator}
    G_tf(x) =\int_\R G(t,x,y) f(y) \, \ddd y
\end{align}
for $f \in B(\R)$ and $x \in \R$.
Furthermore, for any $\frac{1}{2}>\delta>0$ there exists $b_0 \geq \sqrt{t}$ such that for all $b \geq b_0$ and $x\geq b/2$ we have
\begin{align*}
\frac{1}{2} -G_s u(0,\cdot)(x) \geq \delta.
\end{align*}
This implies by Lemma \ref{lemma_noise_bound} and Gaussian tail bounds, choosing some fixed $\delta >0$ and $b$ large enough, that
\begin{align*}
\PP(\tau_b \leq t) &\leq \PP(\exists x \in [b/2,\infty[ ,s\leq t \colon N_s(x) \geq \delta)\\
&\leq C(t,\delta)  \left(\int_{]-\infty, 0]} G(t,b/2,z) \, \ddd z+\int_{[b/2,\infty[} \int_{]-\infty, 0]} G(t,x,z) \, \ddd z \, \ddd x \right)\\
&\leq C(t,\delta) \exp\left(-\frac{b^2}{8t} \right)
\end{align*}
for some $C(t,\delta)>0$ as desired.
\end{proof}
\begin{corol}
In the setting of Proposition \ref{prop_int_ex} we have that for every $t>1$
\begin{align*}
\E\left[\,\sup_{0\leq s \leq t}\abs{ R(v(s,\cdot))} \right] , \E\left[\,\sup_{0\leq s \leq t}\abs{R(u(s,\cdot))} \right] < \infty 
\end{align*}
and
\begin{align*}
\E\left[\,\sup_{0\leq s \leq t} \abs{L(v(s,\cdot))} \right],\E\left[\,\sup_{0\leq s \leq t}\abs{L(u(s,\cdot))} \right] < \infty.
\end{align*}
In particular, we have almost surely
\begin{align*}
\sup_{0\leq s \leq t}\abs{R(v(s,\cdot))}, \sup_{0\leq s \leq t}\abs{R(u(s,\cdot))} <\infty \text{  and   }
\sup_{0\leq s \leq t}\abs{L(v(s,\cdot))}, \sup_{0\leq s \leq t}\abs{L(u(s,\cdot))} <\infty.
\end{align*}
\end{corol}

\begin{proof}
For the right edge, we have by Proposition  \ref{prop_int_ex} that for $b$ large enough
\begin{align*}
    \PP\left(\sup_{0\leq s \leq t} R(u(s,\cdot)) >b\right)\leq \eta(t,b).
\end{align*}
Now, note that $(1-u(t,-x),1-v(t,-x))$ also solves Equation \ref{eq:OnOffStochasticHeat} with initial condition $(u_0,v_0)$. Hence, we also have
\begin{align*}
    \PP\left(\sup_{0\leq s \leq t} R(1-u(s,-\cdot)) >b\right)=\PP\left(-\inf_{0\leq s \leq t} L(u(s,\cdot)) >b\right)\leq \eta(t,b).
\end{align*}
Since $R(u(s,\cdot))\geq L(u(s,\cdot))$, we finally obtain 
\begin{align*}
    \PP\left(\sup_{0\leq s \leq t} \vert R(u(s,\cdot))\vert  >b\right)\leq \eta(t,b),
\end{align*}
implying that
\begin{align*}
    \E\left[\sup_{0\leq s \leq t} \vert R(u(s,\cdot))\vert \right]< \infty.
\end{align*}
By symmetry, we get the analogous result for the left edge.

 For the $v$ component, we have by the delay representation (see Theorem \ref{thm:DelayRepresentation}) that
\begin{align*}
    v(t,x)&=e^{-ct}v_0(x)+\int_0^t ce^{-c(t-s)} u(s,x) \, \ddd s.
\end{align*}
Hence, as we have for $t \geq 0$ and any $x > \sup_{0\leq s\leq t} \vert R(u(s,\cdot))\vert $ that $ v(s,x) =0 $ for every $0\leq s \leq t$, we obtain
\begin{align*}
    \sup_{0\leq s \leq t}  R( v(s,\cdot)) \leq \sup_{0\leq s \leq t} \vert R(u(s,\cdot))\vert.
\end{align*}
Similarly, it follows that
\begin{align*}
    \inf_{0\leq s \leq t}  L( v(s,\cdot)) \geq -\sup_{0\leq s \leq t} \vert L(u(s,\cdot))\vert.
\end{align*}
Combining the preceding two equations, we obtain the desired result for $v$.
\end{proof}

\section{Auxiliary results} \label{section:auxiliary}
\noindent Here, we provide all the calculations required for the preceding section.\\
Define for $t \geq 0$ and $x \in \R$ the quantities
\begin{align*}
D_t(x)&\defeq c\int_0^t \int_{\R} G(t-s,x,y)  u(s,y) \, \ddd y \, \ddd s, \\
E_t (x)& \defeq c\int_0^t \int_{\R} G(t-s,x,y) v(s,y) \, \ddd y \, \ddd s,\\
M_t(x)& \defeq \int_0^t \int_\R G(t-s,x,y) \sqrt{(1-u(s,y))u(s,y)} \, W(\ddd s, \ddd y)
\end{align*}
and note that
$$N_t(x)=D_t(x)+E_t (x)+M_t(x)$$
for $t\geq 0$ and $x \in \R$.
Then, we have:

\begin{lemma} \label{lemma:boundsprereq}
In the setting of Theorem \ref{prop_int_ex} we have for all $p\geq 1$ and $b>1$ the existence of a constant $C(p)>0$ such that for all $0\leq s\leq t$ and $x,y \in [b/2,\infty[$ we have
\begin{align*}
\E\left[\abs{N_t(x)-N_t(y)}^{2p}\right] &\leq C(p) (t^{1/2}(\abs{x-y}\wedge t^{1/2} )^{p-1}+t(t^{1/2}\abs{x-y} \wedge t )^{2p-1})   \\
&\qquad \times\int_{]-\infty, 0]} (G(t,x,z)+G(t,y,z)) \, \ddd z,\\
\E\left[\abs{N_t(x)-N_s(x)}^{2p}\right] &\leq C(p)  (t^{1/2}\abs{t-s}^{(p-1)/2}+t(t\abs{t-s})^{(2p-1)/2}+t\vert t-s\vert^{2p-1})   \\
&\qquad \times\int_{]-\infty, 0]} (G(t,x,z)+G(s,x,z)) \, \ddd z.
\end{align*}
\end{lemma}

\begin{remark}
A similar proposition can be found in \cite{T95} and \cite{M19}. However, the introduction of the seed bank drift term poses technical difficulties which we tackle by using a duality technique.
\end{remark}
\begin{proof}
We only verify the first inequality, the second one can be completed in a similar manner. Note, that the following bound on the heat kernel is well-known for all $t\geq 0\text{ and } x,y \in \R$:
\begin{align*}
\int_0^t \int_\R (G(t-s,x,z)-G(t-s,y,z))^2 \, \ddd z \, \ddd s \leq C(\abs{x-y} \wedge t^{1/2}).
\end{align*}
By the BDG and Hölder inequality and the fact that $0\leq u\leq 1$, we get
\begin{align*}
&\E\left[\abs{M_t^b(x)-M_t^b(y)}^{2p}\right] \\
&\qquad\leq C(p) \E\left[\left( \int_0^t \int_\R (G(t-s,x,z)-G(t-s,y,z))^2 u(s,z)(1-u(s,z))  \, \ddd z \, \ddd s\right)^p\right] \\
&\qquad \leq C(p) (\abs{x-y} \wedge t^{1/2} )^{p-1} \E \left[ \int_0^t \int_\R (G(t-s,x,z)-G(t-s,y,z))^2 u(s,z) \, \ddd z \, \ddd s  \right]\\
&\qquad \leq C(p) (\abs{x-y} \wedge t^{1/2} )^{p-1} \int_0^t (t-s)^{-1/2} \int_\R (G(t-s,x,z)+G(t-s,y,z)) \E[u(s,z)] \, \ddd z \, \ddd s .
\end{align*}
Now, by Theorem \ref{thm:duality}, we have, denoting by $ (B_t)_{t \geq 0}$ an on/off and by $ (\tilde B_t)_{t \geq 0}$ a \textit{standard} Brownian motion,
\begin{align*}
\E[u(s,z)] &=\PP_{(0,\boldsymbol{a})}( B_s \geq z) \\
&= \int_\R \PP(\tilde B_{s-r} \geq z) \,\ddd \PP_J(r)\\
&= \int_\R \int_{]-\infty,0]} G(s-r,z,w) \, \ddd w \, \ddd \PP_J(r)
\end{align*}
for $s\geq 0, z\in \R$. Here, $J$ denotes the random time during which the on/off Brownian motion is switched off, which is independent of the movement of the Brownian motion. Thus, by the semigroup property of the heat kernel, we have
\begin{align*}
&\E\left[\abs{M_t(x)-M_t(y)}^{2p}\right] \\
&\leq C(p) (\abs{x-y} \wedge t^{1/2} )^{p-1} \int_0^t (t-s)^{-1/2} \int_{]-\infty, 0]} \int_\R (G(t-r,x,z)+G(t-r,y,z)) \, \ddd \PP_J(r) \, \ddd z \, \ddd s  \\
&\leq C(p) (\abs{x-y} \wedge t^{1/2} )^{p-1} \int_0^t (t-s)^{-1/2} \int_{]-\infty, 0]} \int_\R (G(t,x,z)+G(t,y,z)) \, \ddd \PP_J(r) \, \ddd z \, \ddd s  \\
&\leq C(p) (\abs{x-y} \wedge t^{1/2} )^{p-1} t^{1/2} \int_{]-\infty, 0]}  G(t,x,z)+G(t,y,z)  \, \ddd z ,  
\end{align*}
where we used that $x,y \geq 0$.

Similarly, using H\"older's inequality and \cite[Lemma 5.2]{M03}, (using $\beta=1$ and $\lambda' =0$ there)
\begin{align*}
&\E\left[\abs{D_t(x)-D_t(y)}^{2p}\right] \\
&\qquad \leq  \left(\int_0^t \int_{\R}\vert G(t-s,x,z)-G(t-s,y,z)\vert  \, \ddd z\, \ddd s \right)^{2p-1} \\
&\qquad \qquad \E\left[ \int_0^t \int_{\R}\vert G(t-s,x,z)-G(t-s,y,z)\vert u(s,z)^{2p} \, \ddd z\, \ddd s  \right]\\
&\qquad \leq C(p) (t^{1/2}\abs{x-y} \wedge t  )^{2p-1}   \E\left[ \int_0^t \int_{\R}(G(t-s,x,z)+G(t-s,y,z)) u(s,z) \, \ddd z\, \ddd s  \right]\\
&\qquad  \leq C(p) (t^{1/2}\abs{x-y} \wedge t )^{2p-1} t  \int_{]-\infty, 0]}  G(t,x,z)+G(t,y,z)  \, \ddd z .
\end{align*}
The exact same calculation works for $E^b_t$ since the only difference in the duality relation is that we start the on/off Brownian motion in the dormant state.
\end{proof}

 This enables us to obtain a bound on the size of $N$.
\begin{lemma} \label{lemma_noise_bound}
In the setting of Theorem \ref{prop_int_ex} we have for all $t>1,b>2 \text{ and } 1 \geq \varepsilon>0$ that there exists some constant $C>0$ such that
\begin{align*}
&\PP\left(\abs{N_s(x)} \geq \varepsilon \text{ for some } x \in ]b/2,\infty[,s\in [0,t]\right)\\
&\qquad \leq C  \varepsilon^{-18} t^{29}  \left(\int_{]-\infty, 0]} G(t,b/2,z) \, \ddd z+\int_{[b/2,\infty[} \int_{]-\infty, 0]} G(t,x,z) \, \ddd z \, \ddd x \right).
\end{align*}
\end{lemma}

\begin{proof}
The proof is basically the same as in \cite[Lemma 3.1]{T95} and \cite[Lemma 23]{Etheridge04}. Consider the dyadic grid given by 
$$G_n \defeq \lbrace (t_{n,i}, x_{n,j})\,\vert \,t_{n,i}=i2^{-n}, x_{n,j}=j2 ^{-n}, i,j \in \N \rbrace , $$
where two points $g^1=(t^1,x^1), g^2=(t ^2, x^2) \in G_n$ are said to be neighboured if $\vert x^1-x^2 \vert =2^{-n}$ and $t^1=t^2$ or vice versa. Now, take neighbouring $g^1, g^2 \in G_n$ with $\vert x^1-x ^2 \vert =2^{-n}$ (we write $g_1 \sim g_2$ for such $g_1,g_2$) and consider for some fixed $\eps_0>0$ and $n\in \N$ the set
$$A^{g^1,g^2}_{\eps_0, n} \defeq \lbrace \vert N_{t^1}(x^1)-N_{t ^1}(x^2) \vert \geq \eps_0 2^{-n/10}\rbrace.$$
Then, using Lemma \ref{lemma:boundsprereq} and Markov's inequality, we obtain for $p>1$
\begin{align*}
    \PP(A^{g^1,g^2}_{\eps_0, n}) &\leq \eps_0^{-2p} 2^{np/5} \E\left[ \vert N_{t^1}(x^1)-N_{t^1}(x^2) \vert^{2p} \right] \\
    &\leq C(p) \eps_0^{-2p} 2^{np/5} \left((t ^1)^{1/2}2^{-n(p-1)}+(t ^1)^{p+1/2}2^{-n(2p-1)}\right) \\
    &\qquad \int_{]-\infty, 0]} G(t ^1,x^1,z) + G(t^1,x^2, z) \, \ddd z\\
    &\leq C(p) \eps_0^{-2p} 2^{np/5} 2^{-n(p-1)} \left((t^1)^{1/2}+(t^1)^{p+1/2}\right) \int_{]-\infty, 0]} G(t ^1,x^1,z) + G(t ^1,x^2, z) \, \ddd z.
\end{align*}
Now, setting
$$A_{\eps_0 }^1 \defeq \bigcup_{n \in \N} \bigcup_{g_1 \sim g_2 \in G_n \cap [0,t] \times [b/2,\infty[ } A^{g^1,g^2}_{\eps_0, n}, $$
we may obtain
\begin{align}
    \PP(A_{\eps_0}^1) &\leq \sum_{ n\in \N} \sum_{i:0\leq t_{n,i}\leq t} \sum_{j:x_{n,j}\geq b/2} \nonumber\\
    &\qquad C(p) \eps_0^{-2p} 2^{np/5} 2^{-n(p-1)} \left((t_{n,i})^{1/2}+(t_{n,i})^{p+1/2}\right) 2\int_{]-\infty, 0]} G(t_{n,i},x_{n,j},z) \, \ddd z \nonumber\\
    &\leq 2 C(p) \eps_0^{-2p} \left(t^{1/2}+t^{p+1/2}\right)\nonumber\\
    &\qquad \sum_{ n\in \N} \sum_{i:0\leq t_{n,i}\leq t} 2^{n(2-4/5p)} \sum_{j:x_{n,j}\geq b/2} 2^{-n} \int_{]-\infty, 0]} G(t_{n,i},x_{n,j},z) \, \ddd z. \label{eq:upperboundsteps}
\end{align}
For the final sum, we proceed by bounding it from above with\footnote{This is possible since $\int_{]-\infty, 0]} G(t_{n,i},x_{n,j},z) \, \ddd z$ is non-increasing in $x_{n,j}\geq b/2$.}
\begin{align*}
    &2^{-n} \int_{]-\infty, 0]} G(t_{n,i},b/2,z) \, \ddd z + \int_{[b/2,\infty[} \int_{]-\infty, 0]} G(t_{n,i},x,z) \, \ddd z \, \ddd x\\
    &\qquad\leq 2^{-n} (t_{n,i}/t)^ {-1/2}\int_{]-\infty, 0]} G(t,b/2,z) \, \ddd z +  (t_{n,i}/t)^ {-1/2} \int_{[b/2,\infty[} \int_{]-\infty, 0]} G(t,x,z) \, \ddd z \, \ddd x.
\end{align*}
Plugging this back into Equation \eqref{eq:upperboundsteps} yields
\begin{align*}
    \PP(A_{\eps_0}^1) &\leq 2 C(p) \eps_0^{-2p} \left(t^{1/2}+t^{p+1/2}\right) \sum_{ n\in \N} 2^{n(3-4/5p)} \\
    &\quad \left(2^{-n} \int_{]-\infty, 0]} G(t,b/2,z) \, \ddd z + \int_{[b/2,\infty[} \int_{]-\infty, 0]} G(t,x,z) \, \ddd z \, \ddd x \right)\sum_{i:0\leq t_{n,i}\leq t} 2 ^{-n} (t_{n,i}/t)^ {-1/2}.
\end{align*}
Then, using that the last sum is bounded from above by 
$$\int_0^t  (s/t)^ {-1/2} \, \ddd s =2 t  $$
and choosing $p=9$, we finally get for some $C>0$
\begin{align*}
    \PP(A_{\eps_0}^1) &\leq  C \eps_0^{-18} \left(t^{3/2}+t^{21/2}\right)  \left(\int_{]-\infty, 0]} G(t,b/2,z) \, \ddd z+\int_{[b/2,\infty[} \int_{]-\infty, 0]} G(t,x,z) \, \ddd z \, \ddd x \right).
\end{align*}
The same bound (with $t^{3/2}+t^{21/2}+t^2$ instead of $t^{3/2}+t^{21/2}$) holds if we replace in the beginning $g^1,g^2$ by neighbouring points in $G_n$ with $\vert t^1-t ^2 \vert =2^{-n}$ to obtain analogously a set $A_{\eps_0}^2$. 

Next, given $0\leq s \leq t$ and $x \geq b/2$ we aim to show that on $(A_{\eps_0}^1\cup A_{\eps_0}^2)^c$ we actually have 
$$\abs{N_s^b(x)} \leq \varepsilon$$
once we choose some specific $\epsilon_0$. For this purpose, choose some $g^0=(t_{1,i},x_{1,j}) \in G_1$ closest to $(s,x)$. From this point we need at most $[t/2^{-1}]+1\leq 3t$ (the $[\cdot]$ here is the Gauss bracket rounding its content to the closest natural number) steps to reach $(0,x_{1,j}) \in G_1$, implying that on $(A_{\eps_0}^1\cup A_{\eps_0}^2)^c$
$$\vert N_{t_{1,i}}(x_{1,j})\vert=\vert N_{t_{1,i}}(x_{1,j})-N_0(x_{1,j})\vert\leq 3t \varepsilon_0 2 ^{-1/10}. $$
Now, as in the proof for the modulus of continuity of Brownian motion, one needs at most one step in time and one in space of length $2^{-n}$ for each $n\geq 2$ to obtain a path from $g^0$ to $(s,x)$. This yields on $(A_{\eps_0}^1\cup A_{\eps_0}^2)^c$
$$\vert N_{s}(x)-N_{t_{1,i}}(x_{1,j})\vert \leq  2 \sum_{n\geq 2} \eps_0 2^{-n/10}. $$
Combining the above, we finally obtain 
$$\vert N_{s}(x) \vert \leq 3t \eps_0 \sum_{n \in \N} 2 ^{-n/10}=C t \eps_0$$
for some $C>0$. Hence, setting for $0<\varepsilon\leq 1$ 
$$\eps_0 \defeq \varepsilon/(Ct),$$
we have on $(A_{\eps_0}^1\cup A_{\eps_0}^2)^c$ 
$$\vert N_{s}(x)\vert \leq \varepsilon. $$
This gives the desired result.
\end{proof}

\begin{lemma} [On/off Feynman-Kac] \label{lemma:onoffFeynmanKac}
Let $(\phi^\lambda,\varphi^\lambda )$ be the solution to the PDE \eqref{eq:Delaysystem} and $t\geq 0$. Then, we have for all $0\leq s \leq t$ and $x \in \R$ the stochastic representation 
\begin{align*}
    \phi^\lambda(s,x)= \E_{(x,\boldsymbol{a})}\left[ \int_{I\cap [0,s ]} \lambda \psi_b(B_r) e^{-\int_{I\cap [0,r]} \phi^\lambda(s-u,B_u) \, \ddd u} \, \ddd r \right],
\end{align*}
where $B=(B_t)_{t \geq 0}$ denotes an on/off Brownian motion starting in an active state and $I \subseteq [0,t]$ the union of random time intervals in which the Brownian path is active.
\end{lemma}

\begin{proof}
Set $J=[0,t]\setminus I$ and consider for $s \in [0,t]$ the quantity
\begin{align*}
   \hat M_s= E_s\phi^\lambda(t-s,B_s)\ind_I(s)+ E_s\varphi^\lambda(t-s,B_s)\ind_J(s),
\end{align*}
where
\begin{align*}
    E_s \defeq \exp\left(-\int_{[0,s]\cap I} \phi^\lambda(t-r,B_r) \, \ddd r \right).
\end{align*}
Then, with an application of the Ito formula on the random time intervals between jumps, we see that for some local martingale $\tilde M=(\tilde M_s)_{0\leq s\leq t}$, after adding and subtracting the compensator of the jumps, that

\begin{align*}
    \hat M_s &=\hat  M_0 +  \int_{[0,s]\cap I}  \left(  \frac{\Delta}{2} \phi^\lambda(t-r,B_r) + \partial_r \phi^\lambda(t-r,B_r) -\phi^\lambda(t-r,B_r)^2 \right) E_r \, \ddd r \\
    &\quad +  c \int_{I\cap [0,s]}( \varphi^\lambda(t-r,B_{r})-\phi^\lambda(t-r,B_{r}))E_r \, \ddd r  \\
    &\quad + c \int_{J\cap [0,s]}( \phi^\lambda(t-r,B_{r})-\varphi^\lambda(t-r,B_{r}))E_r \, \ddd r  \\
    &\quad + \int_{J \cap [0,s]} \partial_r \varphi^\lambda(t-r,B_r) E_r \, \ddd r + \tilde M_s
\end{align*}
for $s\geq 0$.
Hence, since $\phi^\lambda, \varphi^\lambda$ are bounded and solve the system \eqref{eq:Delaysystem}, we see that
\begin{align*}
    \E_{(x,\boldsymbol{a})}[\hat M_s]=\E_{(x,\boldsymbol{a})}[\hat M_0]-\E_{(x,\boldsymbol{a})}\left[ \int_{[0,s] \cap I} \lambda \psi_b(B_r)E_r \, \ddd r\right].
\end{align*}
In particular, for $s=t$ we see, since we start in an active state, that 
\begin{align*}
    0=\E_{(x,\boldsymbol{a})}[M_t]=\phi^\lambda(t,x) - \E_{(x,\boldsymbol{a})}\left[ \int_{[0,t] \cap I} \lambda \psi_b(B_r) E_r \, \ddd r \right].
\end{align*}
This gives the desired result.
\end{proof}

\begin{lemma} \label{lemma:initialbound}
Let $(\phi^\lambda,\varphi^\lambda )$ be the solution to the PDE \eqref{eq:Delaysystem}. Then, for all $s\leq t$, $x < b-\sqrt{t}$ and $\lambda>0$ we have the existence of a constant $K>0$ such that 
\begin{align*}
    \phi^\lambda(s,x) \leq \frac{K}{t} \exp{\left(- \frac{(b-x)^2}{20t} \right)}.
\end{align*}
\end{lemma}

\begin{proof}
This lemma is the on/off version of \cite[Lemma 3.5]{D89}. Set
\begin{align*}
    \tau = \inf \lbrace  t\geq 0 \vert B_t \geq b-\sqrt t \rbrace,
\end{align*}
where $B=(B_t)_{t \geq 0}$ denotes an on/off Brownian motion.
By Lemma \ref{lemma:onoffFeynmanKac}, we have for any $s\leq t$, $\lambda >0$ and $x<b-\sqrt t$, using the strong Markov property and that the support of $\psi_b$ is $]b,\infty[$, that
\begin{align*}
    \phi^\lambda (s,x) &=\E_{(x,\boldsymbol{a})}\left[ \int_{[0,s] \cap I} \lambda \psi_b(B_r) E_r \, \ddd r \right] \\
    &= \E_{(x,\boldsymbol{a})}\left[ \int_{[0,s] \cap I} \lambda \psi_b(B_r) E_r \, \ddd r \ind_{\lbrace\tau \geq s \rbrace}\right] +\E_{(x,\boldsymbol{a})}\left[ \int_{[0,s] \cap I} \lambda \psi_b(B_r) E_r \, \ddd r \ind_{\lbrace\tau \leq s \rbrace}\right] \\
    &= \E_{(x,\boldsymbol{a})}\left[ \int_{[\tau,s] \cap I} \lambda \psi_b(B_r) E_r \, \ddd r \ind_{\lbrace\tau \leq s \rbrace}\right]\\
    &=\E_{(x,\boldsymbol{a})} \left[\ind_{\lbrace\tau \leq s \rbrace} \exp\left( -\int_{[0,\tau]\cap I} \phi^\lambda (s -u ,B_u ) \, \ddd u \right)\right.\\
    &\qquad\quad \times \left.\E\left[ \int_{[\tau,s] \cap I} \lambda \psi_b(B_r) \exp\left( - \int_{[\tau,r]\cap I} \phi^\lambda(s-u,B_u) \, \ddd u \right) \, \ddd r \Bigg\vert \mathcal{F}_\tau\right] \right]\\
    &\leq  \,\E_{(x,\boldsymbol{a})} \left[\ind_{\lbrace\tau \leq s \rbrace}\E_{(B_\tau, \boldsymbol{a})} \left[ \int_{[0,s-\tau] \cap I} \lambda \psi_b(B_r) \exp\left( - \int_{[0,r]\cap I} \phi^\lambda(s-\tau-u,B_u)  \,\ddd u \right)  \ddd r \right] \right] \\
    &=  \E_{(x,\boldsymbol{a})}[\ind_{\lbrace\tau \leq s \rbrace}\phi^\lambda(s-\tau, B_\tau)].
\end{align*}
Now, by the bound \eqref{eq:bound1}, we have, since $B_\tau =b-\sqrt t$,
\begin{align*}
    \E_{(x,\boldsymbol{a})}[\ind_{\lbrace\tau \leq s \rbrace}\phi^\lambda(s-\tau, B_\tau)] \leq \frac{K}{t}\PP_{(x,\boldsymbol{a})}(\tau \leq s) 
\end{align*}
for some constant $K>0$.
Moreover, it holds that
\begin{align*}
    \PP_{(x,\boldsymbol{a})}(\tau \leq s) \leq \PP(\tilde \tau \leq s),
\end{align*}
where $\tilde B=(\tilde B_t)_{t \geq 0}$ is a standard Brownian motion and
$$\tilde \tau = \inf \lbrace t \geq 0 \vert \tilde B_t \geq b-\sqrt{t} \rbrace.$$
Thus, using the reflection principle as in \cite[Proposition 3.2]{T95}, we finally see that
\begin{align*}
    \PP_{(x,\boldsymbol{a})}(\tau \leq s) \leq \PP(\tilde \tau \leq s ) \leq \exp\left(-\frac{(b-x)^2}{20t} \right)
\end{align*}
as desired.
\end{proof}

\begin{lemma} \label{lemma:DelayExistence}
Let $\lambda >0$, $b>0$, $T>0$ and $\psi_b$ be as in the proof of Proposition \ref{prop_int_ex}. The partial functional differential equation given by 
\begin{align} \label{eq:PDE}
	\partial_s \phi^\lambda (s,x) &= \frac{\Delta}{2} \phi^\lambda(s,x) -\frac{1}{4} (\phi^\lambda(s,x))^2 +c \left(ce^{-cs}\int_0^s e^{cs'}\phi^\lambda(s',x) \, \ddd s' -\phi^\lambda(s,x)\right )+ \lambda \psi_b(x)
\end{align}
with initial condition $\phi^\lambda (0,\cdot)\equiv 0$ has a unique $C^{1,2}_b([0,T] \times \R)\cap B([0,T],L^2( \R))$-valued positive solution.
\end{lemma}

\begin{proof}
We begin by considering the \textit{linear} Delay PDE given by
\begin{align}\label{eq:Delaypdelinear}
	\partial_s \phi_1^\lambda (s,x) -\frac{\Delta}{2} \phi_1^\lambda(s,x) &=   c \left(ce^{-cs}\int_0^s e^{cs'}\phi_1^\lambda(s',x) \, \ddd s' -\phi_1^\lambda(s,x)\right )+ \lambda \psi_b(x).
\end{align}
Then, since the right hand side satisfies a linear growth and Lipschitz bound, we get global existence of a solution $\phi_1^\lambda $ taking values in $C^{1,2}_b([0,T] \times \R)\cap B([0,T],L^2( \R))$ by classical theory for Delay PDEs (see e.g. \cite{W96} or simply by Picard iteration). Now, choose for each $\lambda>0$ the map $\phi_1^\lambda$ as an upper and the constant zero map (considered as a solution of the homogeneous version of Equation \eqref{eq:PDE}) as a lower solution. Then, \cite[Theorem 2.1]{P98_2}\footnote{We choose in the setting of the original theorem $J(t-s,x)=c \exp(-c(t-s))$, $\eta(t,x)\equiv h(t,x) \equiv 0$. Since $\phi^\lambda_1$ is bounded, we only need the Lipschitz condition from condition (H1) on a compact interval.} yields existence and uniqueness of the solution $\phi^\lambda$ and
$$0\leq \phi^\lambda(s,x) \leq \phi_1^\lambda (s,x)$$
for all $0\leq s \leq t$ and $ x\in \R$. In particular, we also obtain  that $\phi^\lambda$ takes values in the space $ C^{1,2}_b([0,T] \times \R)\cap B([0,T],L^2( \R))$.
\end{proof}

\bibliographystyle{abbrv}
\bibliography{literature}

\end{document}